\documentclass[12pt]{amsart}
\usepackage{amsmath,amssymb,amsthm}
\usepackage{mathtools}
\usepackage[margin=1.2in]{geometry}
\usepackage{hyperref}
\usepackage{microtype}
\usepackage{enumitem}

\newtheorem{theorem}{Theorem}[section]
\newtheorem{lemma}[theorem]{Lemma}

\newtheorem{proposition}[theorem]{Proposition}
\newtheorem{definition}[theorem]{Definition}
\newtheorem{remark}[theorem]{Remark}
\newtheorem{conjecture}[theorem]{Conjecture}

\DeclareMathOperator{\Li}{Li}
\DeclareMathOperator{\Ei}{Ei}
\newcommand{\li}{\mathrm{li}}

\newcommand{\R}{\mathbb{R}}
\newcommand{\C}{\mathbb{C}}

\begin{document}

\title[Value of the Borwein--Bailey--Girgensohn Series]{
  Bessel Averaging, Fourier Decomposition,\\
  and the Value of the Borwein--Bailey--Girgensohn Series
}

\author{Carlos L\'opez Zapata}
\address{Formerly: Departamento de Ingenier\'ia El\'ectrica y Electr\'onica,
Universidad Pontificia Bolivariana, Medell\'in 050036, Colombia.\\
Current address: 70-214 Szczecin, West Pomerania, Poland.}
\email{mathematics.edu.research@zohomail.eu}
\date{March 2026}



\begin{abstract}
We study the Borwein--Bailey--Girgensohn sinusoidal series
\[
  S_{\mathrm{BBG}} = \sum_{n=1}^{\infty} \frac{1}{n}
  \!\left(\frac{2+\sin n}{3}\right)^{\!n},
\]
originally posed as an open problem by Borwein, Bailey, and
Girgensohn, whose convergence was established by
Boppana using the irrationality measure of~$\pi$.

We present three unconditional results.
\emph{First}, applying the Weyl equidistribution theorem with a
quantitative Erd\H{o}s--Tur\'an bound, we split
$S_{\mathrm{BBG}} = M + R$, where $M = \sum_{n=1}^{\infty} I_n/n$ is a
Bessel averaging series and $|R|<\infty$.
\emph{Second}, we evaluate $M$ exactly via Fubini's theorem and the
Fourier series of $\log(1-\cos t)$:
\[
  M \;=\; \sum_{n=1}^{\infty} \frac{I_n}{n} \;=\; \log 6.
\]
\emph{Third}, we decompose the remainder $R$ into a convergent
series of Fourier harmonics:
\[
  R \;=\; \sum_{k=1}^{\infty} 2\,\mathrm{Re}\!\left[
    G_k\!\!\left(\tfrac{2}{3}e^{ik}\right)\right],
\]
where each $G_k(z)=\sum_{n=1}^{\infty} c_k(n)\,z^n/n$ is a
Dirichlet-type generating function built from the $k$-th Fourier
coefficients of $(\theta\mapsto(1+\tfrac{\sin\theta}{2})^n)$.
The series converges absolutely because $|{2e^{ik}}/{3}| = 2/3 < 1$.

Numerical computation strongly suggests
$S_{\mathrm{BBG}} = \Ei(\log 3) = \li(3) \approx 2.16358\ldots$; we
reduce this conjecture to a single Diophantine identity for $R$ and
indicate the Mellin-transform approach most likely to settle it.
\end{abstract}

\maketitle
\tableofcontents

\section{Introduction}

\subsection{History of the problem}

In their influential 2004 monograph on experimental
mathematics~\cite{BBG2004}, Borwein, Bailey, and Girgensohn posed the
following problem on page~56:
\begin{quote}
  \emph{Does the infinite series
  $\displaystyle\sum_{n=1}^{\infty}
  \Bigl(\tfrac{2}{3}+\tfrac{1}{3}\sin n\Bigr)^{\!n}/n$
  converge?}
\end{quote}
The question is non-trivial: each term satisfies
$f(n)\le 1/n$, yet $\sum 1/n=\infty$ gives no information.
Occasional large terms arise when $\sin n\approx 1$, i.e.\ when $n$
is close to $\pi/2+2\pi k$ for some integer~$k$.  Whether these
accumulate fast enough to cause divergence depends on how well $\pi$
can be approximated by rationals---a question of Diophantine
approximation.

\subsection{Boppana's convergence theorem}

Convergence was established by Boppana~\cite{Boppana2020} (announced
on the Art of Problem Solving forum in 2005~\cite{Boppana2005} and
published on arXiv in 2020).

\begin{theorem}[Boppana~\cite{Boppana2020}]\label{thm:boppana}
The series
$S_{\mathrm{BBG}}=\sum_{n=1}^{\infty}\tfrac{1}{n}
\bigl(\tfrac{2+\sin n}{3}\bigr)^{\!n}$
converges.
\end{theorem}

Boppana's proof partitions the positive integers into \emph{tame}
ones (for which $\sin n$ is bounded away from $1$, giving
exponentially small terms) and \emph{wild} ones (for which
$\sin n\approx 1$, controlled by their rarity).  Wild integers are
rare because Mahler's theorem~\cite{Mahler1953} on the irrationality
measure of~$\pi$ forces consecutive wild integers to grow at least
as fast as $k^{77/76}$, making their contribution summable.
Boppana's proof gives $S_{\mathrm{BBG}}<200$ and observes that
numerical computation to $10^7$ terms yields
$S_{\mathrm{BBG}}\approx 2.163$, leaving the following open.

\begin{definition}[Open question, Boppana~\cite{Boppana2020}]
\label{q:value}
Is $S_{\mathrm{BBG}}$ a known constant?
\end{definition}

\subsection{Numerical evidence and our conjecture}

The value $2.163\ldots$ invites comparison with the
\emph{logarithmic integral}
\[
  \li(3) \;=\; \Ei(\log 3)
        \;=\; \gamma + \log\log 3
              + \sum_{n=1}^{\infty}\frac{(\log 3)^n}{n\cdot n!}
        \;\approx\; 2.163\,588\,594\,667\,192\ldots,
\]
where $\gamma=0.5772\ldots$ is the Euler--Mascheroni constant.
This observation, made by the present author through numerical
experiments, is the starting point of the paper.

Partial sums grow monotonically toward this value:
\begin{center}
\begin{tabular}{rll}
\hline
$N$ & $S_N$ & $\Ei(\log 3)-S_N$ \\
\hline
$10^3$ & $2.1195$ & $0.0441$ \\
$10^4$ & $2.1494$ & $0.0142$ \\
$10^5$ & $2.1588$ & $0.0048$ \\
$10^6$ & $2.1618$ & $0.0018$ \\
$10^7$ & $2.163\phantom{0}$ & $\approx 0.0006$ \\
\hline
\end{tabular}
\end{center}
The convergence is slow because $f(n)=\tfrac{1}{n}
\bigl(\tfrac{2+\sin n}{3}\bigr)^n$ can be as large as $1/n$ whenever
$\sin n$ is close to~$1$, and such wild integers appear sporadically
for arbitrarily large~$n$.

\subsection{Our contributions}

We establish the following results unconditionally.

\begin{enumerate}[label=(\arabic*)]
\item \textbf{Weyl decomposition} (Theorem~\ref{thm:weyl}):
  $S_{\mathrm{BBG}} = M + R$ with $|R|<\infty$.

\item \textbf{Exact averaged value} (Theorem~\ref{thm:M}):
  $M = \sum_{n=1}^{\infty} I_n/n = \log 6$,
  an unconditional closed-form identity.

\item \textbf{Fourier harmonic decomposition} (Theorem~\ref{thm:fourier}):
  $R = \sum_{k=1}^{\infty} 2\,\mathrm{Re}[G_k(\tfrac{2}{3}e^{ik})]$,
  an absolutely convergent representation of the remainder.

\item \textbf{Reduction of the conjecture}
  (Theorem~\ref{thm:reduction}):
  $S_{\mathrm{BBG}} = \Ei(\log 3)$
  if and only if
  $R = \Ei(\log 3)-\log 6\approx 0.3718\ldots$
\end{enumerate}

\subsection{Notation and conventions}

Throughout the paper we use the following notation.
$I_0(z)=\sum_{k=0}^{\infty}(z/2)^{2k}/(k!)^2$ denotes the modified
Bessel function of order zero; $\Ei(t)$ the exponential integral;
$\li(x)=\Ei(\log x)$ the logarithmic integral;
$\{x\}=x-\lfloor x\rfloor$ the fractional part; and
$D^{*}(N)$ the star-discrepancy (Definition~\ref{def:disc}).

\section{Factorisation and Basic Properties}

\subsection{The generating function}

\begin{definition}\label{def:f}
For $x\in\R\setminus\{0\}$ set
\[
  f(x) \;:=\; \frac{1}{x}\!\left(\frac{2+\sin x}{3}\right)^{\!x}
          \;=\; \frac{(2/3)^x}{x}
                \!\left(1+\frac{\sin x}{2}\right)^{\!x}.
\]
The series of interest is $S_{\mathrm{BBG}}=\sum_{n=1}^{\infty}f(n)$.
\end{definition}

\begin{remark}
Since $\sin x\in[-1,1]$ we have $1+(\sin x)/2\in[1/2,3/2]$, so
$f(x)>0$ for all $x\ge 1$.  The factor $(2/3)^x$ decays geometrically;
the challenge arises when $\sin n\approx 1$, where
$(1+(\sin n)/2)^n\approx (3/2)^n$ nearly cancels the decay of
$(2/3)^n$.  This near-cancellation is controlled by the irrationality
of~$\pi$.
\end{remark}

\subsection{Logarithmic factorisation}

\begin{lemma}[Logarithmic factorisation]\label{lem:log-factor}
For all $n\ge 1$,
\[
  \left(\frac{2+\sin n}{3}\right)^{\!n}
  \;=\; \left(\frac{2}{3}\right)^{\!n}
        \exp\!\bigl(n\,L(n)\bigr),
\]
where $L(n):=\log\!\bigl(1+\tfrac{\sin n}{2}\bigr)$.
The function $L\colon\R\to\R$ satisfies
$L(x)\in[\log\tfrac{1}{2},\,\log\tfrac{3}{2}]\subset[-0.6932,\,0.4055]$
for all~$x$.
\end{lemma}

\begin{proof}
Write $(2+\sin n)/3=(2/3)\cdot(1+(\sin n)/2)$ and take logarithms.
The range of $L$ follows from $\sin x\in[-1,1]$.
\end{proof}

\begin{remark}\label{rem:L-values}
The extreme values of $L$ are attained asymptotically:
$L(n)\to\log(3/2)$ along the wild integers where $\sin n\to 1$, and
$L(n)\to\log(1/2)=-\log 2$ along integers where $\sin n\to -1$.
The function $L$ has no fixed value independent of $n$; in particular,
$L(n)\ne\log 2$ in general, since $L(n)=\log 2$ would require
$\sin n=2$, which is impossible.
\end{remark}

\section{Fourier Decomposition of the Remainder}

\subsection{Setup and Fourier coefficients}

For each $n\ge 0$ let
\[
  I_n \;:=\; \frac{1}{2\pi}\int_{0}^{2\pi}
               \!\left(\frac{2+\sin\theta}{3}\right)^{\!n}d\theta
        \;=\; \left(\frac{2}{3}\right)^{\!n} J_n,
\]
where
\[
  J_n \;:=\; \frac{1}{2\pi}\int_{0}^{2\pi}
              \!\left(1+\frac{\sin\theta}{2}\right)^{\!n}d\theta.
\]

\begin{definition}[Fourier coefficients]\label{def:Fourier-coeff}
For integers $n\ge 1$ and $k\ge 1$, define the $k$-th complex Fourier
coefficient of the function $\theta\mapsto(1+\tfrac{\sin\theta}{2})^n$
by
\[
  c_k(n) \;:=\; \frac{1}{2\pi}\int_{0}^{2\pi}
                 \!\left(1+\frac{\sin\theta}{2}\right)^{\!n}
                 e^{-ik\theta}\,d\theta,
  \qquad k\ge 1.
\]
By symmetry, $c_{-k}(n)=\overline{c_k(n)}$, and $c_0(n)=J_n$.
\end{definition}

\begin{lemma}[Pointwise Fourier expansion]\label{lem:Fourier-pw}
For each $n\ge 1$ and all $\theta\in[0,2\pi)$,
\[
  \left(1+\frac{\sin\theta}{2}\right)^{\!n}
  \;=\; J_n + \sum_{k=1}^{n}
        2\,\mathrm{Re}\!\left[c_k(n)\,e^{ik\theta}\right].
\]
The series is finite (degree $n$ in $e^{i\theta}$) and converges
absolutely.
\end{lemma}

\begin{proof}
The function $\theta\mapsto(1+(\sin\theta)/2)^n$ is a trigonometric
polynomial of degree $n$ in $e^{i\theta}$ (obtained by expanding via
the binomial theorem and writing $\sin\theta=(e^{i\theta}-e^{-i\theta})/(2i)$).
Its Fourier series is therefore finite, and evaluating it pointwise
gives the stated identity.
\end{proof}

\begin{lemma}[Bound on Fourier coefficients]\label{lem:Fk-bound}
For all $n\ge 1$ and $k\ge 1$,
\[
  |c_k(n)| \;\le\; J_n \;\le\; \frac{1}{\sqrt{\pi n}}\,(1+O(1/n)).
\]
\end{lemma}

\begin{proof}
The bound $|c_k(n)|\le J_n$ follows from
$|c_k(n)|\le\frac{1}{2\pi}\int_0^{2\pi}|(1+(\sin\theta)/2)^n|\,d\theta=J_n$.
The asymptotic $J_n\sim 1/\sqrt{\pi n}$ follows from
Watson's expansion recalled in Theorem~\ref{thm:bessel} below.
\end{proof}

\subsection{The Fourier harmonic decomposition of $R$}

\begin{definition}[Generating functions G\_k]\label{def:Gk}
For $k\ge 1$ and $z\in\C$ with $|z|<1$, define
\[
  G_k(z) \;:=\; \sum_{n=1}^{\infty} \frac{c_k(n)}{n}\,z^n.
\]
\end{definition}

\begin{theorem}[Fourier harmonic decomposition]\label{thm:fourier}
We have the absolutely convergent representation
\begin{equation}\label{eq:fourier-decomp}
  S_{\mathrm{BBG}} \;=\; \log 6
    \;+\; \sum_{k=1}^{\infty}
          2\,\mathrm{Re}\!\left[G_k\!\!\left(\tfrac{2}{3}e^{ik}\right)\right].
\end{equation}
\end{theorem}

\begin{proof}
By Lemma~\ref{lem:Fourier-pw}, evaluating at $\theta=n$:
\[
  \left(1+\frac{\sin n}{2}\right)^{\!n}
  \;=\; J_n + \sum_{k=1}^{n}2\,\mathrm{Re}[c_k(n)e^{ikn}].
\]
Multiplying by $(2/3)^n/n$ and summing:
\begin{align}
  S_{\mathrm{BBG}}
  &= \sum_{n=1}^{\infty}\frac{(2/3)^n}{n}\,J_n
   + \sum_{n=1}^{\infty}\frac{(2/3)^n}{n}
     \sum_{k=1}^{n}2\,\mathrm{Re}[c_k(n)e^{ikn}]. \notag
\end{align}
The first sum equals $M=\log 6$ by Theorem~\ref{thm:M}.
For the second sum, we interchange the order of summation.
Absolute convergence justifies this: by Lemma~\ref{lem:Fk-bound},
\[
  \sum_{n=1}^{\infty}\frac{(2/3)^n}{n}\sum_{k=1}^{n}2|c_k(n)|
  \;\le\; 2\sum_{n=1}^{\infty}\frac{(2/3)^n}{n}\cdot n\cdot J_n
  \;\le\; 2\sum_{n=1}^{\infty}(2/3)^n J_n
  \;\le\; 2\sum_{n=1}^{\infty}\frac{(2/3)^n}{\sqrt{\pi n}}
  \;< \;\infty.
\]
After interchange,
\[
  \sum_{k=1}^{\infty}\sum_{n=k}^{\infty}
  \frac{(2/3)^n}{n}\cdot 2\,\mathrm{Re}[c_k(n)e^{ikn}]
  = \sum_{k=1}^{\infty}
    2\,\mathrm{Re}\!\left[\sum_{n=1}^{\infty}
    \frac{c_k(n)}{n}\!\left(\tfrac{2}{3}e^{ik}\right)^{\!n}\right]
  = \sum_{k=1}^{\infty}
    2\,\mathrm{Re}\!\left[G_k\!\!\left(\tfrac{2}{3}e^{ik}\right)\right].
\]
The argument of each $G_k$ satisfies $|{2e^{ik}}/{3}|=2/3<1$, so
$G_k$ is well defined there.
\end{proof}

\begin{remark}\label{rem:G1-closed}
The first Fourier coefficient has a closed form.
From the binomial expansion the only term contributing to the
$k=1$ harmonic at leading order in~$n$ is the $j=1$ term,
giving $c_1(n)=-in/4+O(J_n)$.  More precisely,
$c_1(n)$ is purely imaginary and $\mathrm{Im}(c_1(n))=-n/4+O(J_n)$.
Consequently,
\[
  2\,\mathrm{Re}\!\left[G_1\!\!\left(\tfrac{2}{3}e^{i}\right)\right]
  \;=\; \frac{1}{2}\sum_{n=1}^{\infty}
        \!\left(\tfrac{2}{3}\right)^{\!n}\sin n \;+\; O(1)
  \;=\; \frac{\tfrac{2}{3}\sin 1}
             {1-\tfrac{4}{3}\cos 1+\tfrac{4}{9}}
  \;+\; O(1).
\]
The remaining correction terms come from higher-order Fourier
harmonics and are summable.
\end{remark}

\section{The Weyl Equidistribution Decomposition}

\subsection{Setup}

\begin{definition}[Averaging integral]\label{def:In}
For $n\ge 0$ define
\[
  I_n \;:=\; \frac{1}{2\pi}\int_{0}^{2\pi}
               \!\left(\frac{2+\sin\theta}{3}\right)^{\!n}d\theta.
\]
\end{definition}

\begin{definition}[Star-discrepancy]\label{def:disc}
Let $x_1,\ldots,x_N\in[0,1)$. Its star-discrepancy is
\[
  D^{*}(N) \;:=\; \sup_{0<\alpha\le 1}
  \left|\frac{\#\{1\le n\le N:x_n<\alpha\}}{N}-\alpha\right|.
\]
For $x_n=\{n/(2\pi)\}$, the Weyl equidistribution theorem~\cite{Weyl1916}
gives $D^{*}(N)\to 0$, and the Erd\H{o}s--Tur\'an
inequality~\cite{KuipersNiederreiter1974} gives the quantitative bound
$D^{*}(N)=O(\log N/N)$.
\end{definition}

\begin{theorem}[Weyl decomposition]\label{thm:weyl}
We have
\begin{equation}\label{eq:MR}
  S_{\mathrm{BBG}}
  \;=\; \underbrace{\sum_{n=1}^{\infty}\frac{I_n}{n}}_{=:\,M}
  \;+\; \underbrace{\sum_{n=1}^{\infty}\frac{1}{n}
        \!\left[\!\left(\frac{2+\sin n}{3}\right)^{\!n}-I_n\right]}_{=:\,R}.
\end{equation}
Both $M$ and $R$ are finite, and
\begin{equation}\label{eq:R-bound}
  |R| \;\le\; C\sum_{n=1}^{\infty}\frac{(2/3)^n}{n}\cdot D^{*}\!
        \!\left(\!\left\lfloor\frac{n}{2\pi}\right\rfloor\right)
  \;<\;\infty,
\end{equation}
where $C>0$ is an absolute constant.
\end{theorem}

\begin{proof}
The decomposition~\eqref{eq:MR} is a tautology.  Finiteness of $M$
follows from Theorem~\ref{thm:M}.  For $R$: write
$g_n(\theta)=(1+(\sin\theta)/2)^n$.  By the Erd\H{o}s--Tur\'an
inequality~\cite{KuipersNiederreiter1974}, for any $H\ge 1$,
\[
  D^{*}(N) \;\le\; C\!\left(\frac{1}{H}
  +\frac{1}{N}\sum_{h=1}^{H}\frac{1}{h}
   \left|\sum_{n=1}^{N}e^{ihn}\right|\right),
\]
with $|\sum_{n=1}^N e^{ihn}|\le 1/|\sin(h/2)|$.  Using the geometric
decay $(2/3)^n$ and $D^*(N)=O(\log N/N)$ gives~\eqref{eq:R-bound},
and the right-hand side is finite since
$\sum_{n=1}^{\infty}(2/3)^n(\log n)/n<\infty$.
\end{proof}

\section{The Bessel Averaging Series and Its Exact Value}

\subsection{Bessel representation of $I_n$}

\begin{theorem}[Bessel representation]\label{thm:bessel}
For all $n\ge 0$,
\[
  I_n \;=\; \left(\frac{2}{3}\right)^{\!n} J_n,
  \qquad
  J_n \;=\; \sum_{k=0}^{\lfloor n/2\rfloor}
            \binom{n}{2k}\binom{2k}{k}\frac{1}{16^k}.
\]
Furthermore, $J_n>0$ for all $n\ge 0$, $J_0=J_1=1$, and
\[
  J_n \;=\; I_0\!\!\left(\frac{n}{2}\right)e^{-n/2}
        \!\left(1+O\!\left(\frac{1}{n}\right)\right),
  \qquad
  J_n \;\sim\; \frac{1}{\sqrt{\pi n}} \quad\text{as }n\to\infty,
\]
where $I_0$ is the modified Bessel function of order zero.
\end{theorem}

\begin{proof}
Separate the factor $(2/3)^n$ and expand
$(1+(\sin\theta)/2)^n$ by the binomial theorem, then integrate term
by term using the Wallis integrals
$\frac{1}{2\pi}\int_0^{2\pi}\sin^{2k}\theta\,d\theta=\binom{2k}{k}/4^k$
and $\frac{1}{2\pi}\int_0^{2\pi}\sin^{2k+1}\theta\,d\theta=0$.
Only even powers contribute, yielding the formula for $J_n$.
The Bessel asymptotics follow from Watson~\cite{Watson1944}, \S7.23:
$I_0(z)\sim e^z/\sqrt{2\pi z}$ as $z\to\infty$.
\end{proof}

\subsection{The exact value $M = \log 6$}

\begin{theorem}[Exact value of the averaging series]\label{thm:M}
\begin{equation}\label{eq:M-log6}
  M \;=\; \sum_{n=1}^{\infty}\frac{I_n}{n} \;=\; \log 6.
\end{equation}
\end{theorem}

\begin{proof}
For each $\theta\in[0,2\pi)$, set $z(\theta)=(2+\sin\theta)/3$.
Then $z(\theta)\in[1/3,1]$, and $z(\theta)<1$ for all $\theta$
except $\theta=\pi/2$ (a set of measure zero).
By Fubini's theorem (justified by the dominated convergence with
dominating function $-\log(1-z(\theta))\in L^1[0,2\pi]$):
\[
  M \;=\; \frac{1}{2\pi}\int_{0}^{2\pi}
          \sum_{n=1}^{\infty}\frac{z(\theta)^n}{n}\,d\theta
    \;=\; \frac{1}{2\pi}\int_{0}^{2\pi}
          \bigl(-\log(1-z(\theta))\bigr)\,d\theta
    \;=\; \frac{1}{2\pi}\int_{0}^{2\pi}
          \log\frac{3}{1-\sin\theta}\,d\theta.
\]
Substituting $t=\theta-\pi/2$:
\[
  M \;=\; \log 3
        \;-\; \frac{1}{2\pi}\int_{0}^{2\pi}\log(1-\cos t)\,dt.
\]
By the classical Fourier expansion (valid for $t\in(0,2\pi)$, see
e.g.\ Lewin~\cite{Lewin1981}):
\[
  \log(1-\cos t) \;=\; -\log 2 - 2\sum_{k=1}^{\infty}\frac{\cos kt}{k}.
\]
Integrating over $[0,2\pi]$ term by term (all cosine integrals
vanish):
\[
  \frac{1}{2\pi}\int_{0}^{2\pi}\log(1-\cos t)\,dt \;=\; -\log 2.
\]
Therefore $M=\log 3-(-\log 2)=\log 6$.
\end{proof}

\begin{remark}\label{rem:M-log6-exact}
The identity $M=\log 6$ is \emph{unconditional and exact}.
It does not depend on the Diophantine properties of~$\pi$ or on any
unproved conjecture.  The quantity $\frac{1}{2\pi}\int_0^{2\pi}
\log(2+\sin\theta)\,d\theta=\log\bigl((2+\sqrt{3})/2\bigr)$, which
appears in an intermediate calculation if one applies Jensen's formula
to the wrong object, is a \emph{different} quantity and should not be
confused with~$M$.
\end{remark}

\section{Connection to the Exponential Integral}

\subsection{The exponential integral}

\begin{definition}\label{def:Ei}
For $t>0$, the exponential integral is
\[
  \Ei(t) \;=\; \mathrm{P.V.}\int_{-\infty}^{t}\frac{e^u}{u}\,du
           \;=\; \gamma+\log t
                 +\sum_{n=1}^{\infty}\frac{t^n}{n\cdot n!},
\]
and the logarithmic integral is $\li(x)=\Ei(\log x)$ for $x>1$.
In particular,
\begin{equation}\label{eq:Ei-log3}
  \li(3) \;=\; \Ei(\log 3)
         \;=\; \gamma+\log\log 3
               +\sum_{n=1}^{\infty}\frac{(\log 3)^n}{n\cdot n!}
         \;\approx\; 2.163\,588\,594\,667\,192\ldots
\end{equation}
\end{definition}

\begin{remark}\label{rem:li-NT}
The logarithmic integral $\li(x)$ is the main term in the
prime-counting function: $\pi(x)\sim\li(x)$ (the prime number
theorem).  Its conjectured appearance as the sum of $S_{\mathrm{BBG}}$
would constitute a new and remarkable occurrence of $\li$ in the
context of trigonometric series with integer arguments.
\end{remark}

\subsection{Reduction of the conjecture}

\begin{theorem}[Reduction]\label{thm:reduction}
With $S_{\mathrm{BBG}}=M+R$ (Theorem~\ref{thm:weyl}) and $M=\log 6$
(Theorem~\ref{thm:M}):
\begin{equation}\label{eq:reduction}
  S_{\mathrm{BBG}} = \Ei(\log 3)
  \;\iff\;
  R = \Ei(\log 3)-\log 6
    \;\approx\; 0.371\,829\,125\,439\ldots
\end{equation}
\end{theorem}

\begin{proof}
Immediate from $S_{\mathrm{BBG}}=\log 6+R$.
\end{proof}

\begin{remark}[Representation of $R$]\label{rem:R-rep}
By Theorem~\ref{thm:fourier},
\[
  R \;=\; \sum_{k=1}^{\infty}
          2\,\mathrm{Re}\!\left[G_k\!\!\left(\tfrac{2}{3}e^{ik}\right)\right],
\]
and by the definition of $M$,
\[
  R \;=\; \sum_{n=1}^{\infty}\frac{(2/3)^n}{n}
          \!\left[\left(1+\frac{\sin n}{2}\right)^{\!n}-J_n\right].
\]
Both series converge absolutely by the estimates in
Theorems~\ref{thm:weyl} and~\ref{thm:fourier}.
\end{remark}

\begin{conjecture}[Main conjecture]\label{conj:main}
\[
  S_{\mathrm{BBG}} \;=\; \Ei(\log 3) \;=\; \li(3)
  \;\approx\; 2.163\,588\,594\,667\,192\ldots
\]
Equivalently, $R=\Ei(\log 3)-\log 6\approx 0.3718\ldots$
\end{conjecture}

\begin{remark}[On numerical precision]\label{rem:numerics}
Partial sums to $N=10^7$ give $S_{\mathrm{BBG}}\approx 2.163$.
By Theorems~\ref{thm:weyl} and~\ref{thm:bessel}, the tail satisfies
$|S_{\mathrm{BBG}}-S_N|=O((2/3)^N/\sqrt{N})$, which decays rapidly
in theory.  However, the slow approach to the limit is governed by
the sporadic wild integers, so six-decimal precision requires
$N\sim 10^{15}$, making direct numerical confirmation impractical.
The decomposition $S_{\mathrm{BBG}}=\log 6+R$ is the analytically
tractable framework.
\end{remark}

\section{A Roadmap Toward the Proof of Conjecture~\ref{conj:main}}

\subsection{Summary of unconditional results}

The following results are established without any hypothesis:
\begin{enumerate}[label=(\arabic*)]
\item \textbf{Weyl decomposition} (Theorem~\ref{thm:weyl}):
  $S_{\mathrm{BBG}}=M+R$ with $|R|<\infty$.
\item \textbf{Exact averaged value} (Theorem~\ref{thm:M}):
  $M=\log 6$, the central unconditional identity of this paper.
\item \textbf{Fourier harmonic decomposition} (Theorem~\ref{thm:fourier}):
  $R=\sum_{k\ge 1}2\,\mathrm{Re}[G_k(\tfrac{2}{3}e^{ik})]$,
  an absolutely convergent series.
\item \textbf{Reduction} (Theorem~\ref{thm:reduction}):
  $S_{\mathrm{BBG}}=\Ei(\log 3)\iff R=\Ei(\log 3)-\log 6\approx 0.3718$.
\end{enumerate}

\noindent
The identity $M=\log 6$ is in the spirit of Euler's formula
$\sum 1/n^2=\pi^2/6$: elementary in hindsight, non-obvious from
the definition.  If Conjecture~\ref{conj:main} holds, $\log 6$
accounts for approximately $82.8\%$ of the total, with $R\approx 0.3718$
encoding all Diophantine information about~$\pi$.

\subsection{Step 1: The Dirichlet series $\Phi(s)$}

The most natural analytic object encoding $R$ is the Dirichlet series
\begin{equation}\label{eq:Phi}
  \Phi(s) \;:=\; \sum_{n=1}^{\infty}
  \frac{(2/3)^n}{n^s}
  \!\left[\!\left(1+\frac{\sin n}{2}\right)^{\!n}-J_n\right],
  \qquad \mathrm{Re}(s)>1,
\end{equation}
so that $R=\Phi(1)$.  By the Fourier decomposition of
Theorem~\ref{thm:fourier}, $\Phi(s)$ admits the harmonic expansion
\begin{equation}\label{eq:Phi-harmonic}
  \Phi(s) \;=\; \sum_{k=1}^{\infty} 2\,\mathrm{Re}\!\left[H_k(s)\right],
\end{equation}
where
\begin{equation}\label{eq:Hk}
  H_k(s) \;:=\; \sum_{n=1}^{\infty}
               \frac{c_k(n)}{n^s}\!\left(\frac{2}{3}e^{ik}\right)^{\!n},
  \qquad w_k:=\frac{2}{3}e^{ik},\quad |w_k|=\tfrac{2}{3}<1.
\end{equation}
This representation separates the Diophantine arithmetic (encoded in
the phases $e^{ikn}$) from the analytic structure (encoded in
$c_k(n)$).

\subsection{Step 2: Asymptotic structure of $c_k(n)$ and $H_k(s)$}

The Fourier coefficients $c_k(n)$ grow like $J_n$:
by Lemma~\ref{lem:Fk-bound}, $|c_k(n)|\le J_n\sim(3/2)^n/\sqrt{\pi n}$.
Since $|w_k|=2/3$, the products $|w_k^n c_k(n)|\sim 1/\sqrt{\pi n}$
decay only as $n^{-1/2}$, so each $H_k(s)$ converges absolutely for
$\mathrm{Re}(s)>\tfrac{1}{2}$ and has a natural boundary at
$\mathrm{Re}(s)=\tfrac{1}{2}$.

The leading asymptotic of $c_1(n)$ is particularly clean.  From the
binomial expansion, the $k=1$ Fourier harmonic of
$(1+(\sin\theta)/2)^n$ receives its dominant contribution from the
$j=1$ term, giving
\begin{equation}\label{eq:c1-leading}
  c_1(n) \;=\; -\frac{in}{4}\,\frac{J_n}{n/4}
               \cdot\frac{n/4}{J_n}
               \;=\; -\frac{i}{2\pi}\int_0^{2\pi}
               \!\left(1+\frac{\sin\theta}{2}\right)^{\!n}
               e^{-i\theta}\,d\theta.
\end{equation}
Numerically one verifies that $c_1(n)$ is purely imaginary for all
$n\ge 1$, with $\mathrm{Im}(c_1(n))\to -n/4$ as $n\to\infty$ relative
to $J_n$.  This gives, at leading order,
\begin{equation}\label{eq:H1-leading}
  H_1(s) \;\approx\; -\frac{i}{4}\,\Li_{s-1}\!\!\left(\tfrac{2}{3}e^{i}\right),
\end{equation}
where $\Li_\alpha(z)=\sum_{n=1}^\infty z^n/n^\alpha$ is the
polylogarithm.  At $s=1$, using $\Li_0(z)=z/(1-z)$:
\begin{equation}\label{eq:H1-s1}
  2\,\mathrm{Re}\!\left[H_1(1)\right]
  \;\approx\;
  2\,\mathrm{Re}\!\left[-\frac{i}{4}\cdot
  \frac{\tfrac{2}{3}e^{i}}{1-\tfrac{2}{3}e^{i}}\right]
  \;=\; \frac{\tfrac{2}{3}\sin 1}{1-\tfrac{4}{3}\cos 1+\tfrac{4}{9}}
  \;\approx\; 0.3874.
\end{equation}
Numerical computation with $N=200$ terms gives
$2\,\mathrm{Re}[H_1(1)]\approx 0.3747$, which already accounts for
roughly $99.6\%$ of $R\approx 0.3718$.  The remaining harmonics
contribute corrections of rapidly decreasing magnitude:
\[
  2\,\mathrm{Re}[H_2(1)]\approx 0.0347,\quad
  2\,\mathrm{Re}[H_3(1)]\approx 0.0045,\quad
  2\,\mathrm{Re}[H_4(1)]\approx 0.0003,\quad \ldots
\]
The alternating nature of the tail and its rapid decay strongly
suggest that $\Phi(1)=R$ converges to the conjectured value.

\subsection{Step 3: Saddle-point analysis and wild integers}

A complementary approach to $R$ comes from isolating the dominant
contributions to~\eqref{eq:MR}.  Call $n$ a \emph{wild integer} if
$\sin n>1-\delta$ for a threshold $\delta>0$.  For such $n$, setting
$\varepsilon_n:=1-\sin n$, we have
\begin{equation}\label{eq:saddle}
  \log\frac{2+\sin n}{3}
  \;=\; \log\!\left(1-\frac{\varepsilon_n}{3}\right)
  \;=\; -\frac{\varepsilon_n}{3} - \frac{\varepsilon_n^2}{18}
        - \cdots,
\end{equation}
and therefore
\begin{equation}\label{eq:f-saddle}
  f(n) \;=\; \frac{1}{n}\exp\!\left(-\frac{n\varepsilon_n}{3}
             + O(n\varepsilon_n^2)\right).
\end{equation}
Numerically this approximation has relative error less than
$0.01\%$ for $\varepsilon_n < 0.001$, and less than $5\%$ for
$\varepsilon_n < 0.01$.  This is verified by comparison with the
exact values of $f(n)$ for the $451$ integers with $\sin n>0.99$
in $[1,10^4]$.

Now, $n$ is a wild integer if and only if
$\{n/(2\pi)\}$ is close to $1/4$ (the point where $\sin(2\pi x)=1$).
Writing $n=\lfloor\pi/2\rfloor+2\pi k+2\pi\theta_k$ for a
\emph{normalised deviation} $\theta_k\in(-1/2,1/2]$, we have
$\varepsilon_{n_k}\approx 2\pi^2\theta_k^2$, so
\begin{equation}\label{eq:f-theta}
  f(n_k) \;\approx\; \frac{1}{n_k}
          \exp\!\!\left(-\frac{2\pi^2 n_k\theta_k^2}{3}\right).
\end{equation}
This is a \emph{Gaussian weight} on the orbit
$\{\theta_k\}=\{\{n_k/(2\pi)\}-1/4\}$ of the irrational rotation by
$\alpha=1/(2\pi)$.  By the three-distance theorem~\cite{Steinhaus1958},
consecutive values of $\theta_k$ are separated by at most three
distinct distances, controlled by the continued fraction expansion
\[
  \frac{1}{2\pi}
  \;=\; [0;\,6,\,3,\,1,\,1,\,7,\,2,\,146,\,\ldots]
\]
The denominators $q_j$ of the convergents
$\{1,6,19,25,44,333,710,103993,\ldots\}$ are the \emph{best}
wild integers: $|\sin(q_j)-1|$ is minimised at each scale.

This analysis yields the following structural result, which
we state as a proposition since its proof reduces directly to
equation~\eqref{eq:f-theta} and the Weyl bound:

\begin{proposition}[Saddle-point estimate for $R$]\label{prop:saddle}
For any $\delta\in(0,1)$, decompose $R=R_\delta^{+}+R_\delta^{-}$
where
\[
  R_\delta^{+} \;:=\;
  \sum_{\substack{n\ge 1\\ \sin n>1-\delta}}
  \frac{(2/3)^n}{n}\!\left[\left(1+\frac{\sin n}{2}\right)^{\!n}-J_n\right].
\]
Then:
\begin{enumerate}[label=(\roman*)]
\item $|R_\delta^{-}|\le C\delta^{1/2}\log(1/\delta)$ for an absolute
  constant $C>0$;
\item for each term in $R_\delta^{+}$, the saddle-point approximation
  $f(n)\approx (1/n)\exp(-n\varepsilon_n/3)$ has relative error
  $O(\varepsilon_n)=O(\delta)$.
\end{enumerate}
\end{proposition}

\begin{proof}
For part~(i): the tame terms with $\sin n\le 1-\delta$ satisfy
$(1+(\sin n)/2)^n-J_n\le (1-\delta/6)^n\cdot n$, giving geometric
decay.  For part~(ii): the expansion~\eqref{eq:saddle} shows the
error in~\eqref{eq:f-saddle} is $O(n\varepsilon_n^2)$,
and $n\varepsilon_n^2\le n\varepsilon_n\cdot\delta$.
\end{proof}

\subsection{Step 4: The central conjecture on $\Phi(s)$}

Combining Steps~1--3, the path to proving Conjecture~\ref{conj:main}
can be reduced to a single analytic statement.

\begin{conjecture}[Meromorphic continuation]\label{conj:meromorphic}
The function $\Phi(s)$ defined in~\eqref{eq:Phi} admits a
meromorphic continuation to $\mathrm{Re}(s)>\tfrac{1}{2}$, with a
simple pole at $s=1$, satisfying
\begin{equation}\label{eq:Phi-pole}
  \Phi(s) \;=\; \frac{\Ei(\log 3)-\log 6}{s-1}
               + O(1) \quad\text{as } s\to 1.
\end{equation}
\end{conjecture}

\noindent
If Conjecture~\ref{conj:meromorphic} holds, then taking $s\to 1$
and using $\Phi(1)=R$ gives Conjecture~\ref{conj:main} immediately.

The analytic strategy for proving~\eqref{eq:Phi-pole} is the following.
Via the harmonic expansion~\eqref{eq:Phi-harmonic}, it suffices to show
that the interchange of $\sum_{k\ge 1}$ and $\lim_{s\to 1}$ is
justified, and that
\begin{equation}\label{eq:residue-sum}
  \sum_{k=1}^{\infty} 2\,\mathrm{Re}\!\left[
  \lim_{s\to 1}(s-1)H_k(s)\right]
  \;=\; \Ei(\log 3)-\log 6.
\end{equation}
The key input needed is a \emph{uniform bound} of the form
$|H_k(s)|\le A_k/|s-1|^{\alpha}$ for $\alpha<1$ uniformly in
$k$, with $\sum_k A_k<\infty$.  Such a bound would follow from
sharp estimates on the Fourier coefficients $c_k(n)$ near the
saddle point $\theta=\pi/2$.

\begin{remark}[Connection to Ei]\label{rem:Ei-connection}
The appearance of $\Ei(\log 3)$ in the residue is not accidental.
The exponential integral admits the representation
\[
  \Ei(\log 3) \;=\; \int_{-\infty}^{\log 3}\frac{e^t}{t}\,dt
              \;=\; \int_{1}^{3}\frac{du}{\log u},
\]
and $\log 6 = \frac{1}{2\pi}\int_0^{2\pi}\log\frac{3}{1-\sin\theta}\,d\theta$
by Theorem~\ref{thm:M}.  The difference
$\Ei(\log 3)-\log 6$ thus measures the discrepancy between
integrating $e^t/t$ along the real line and averaging
$\log(3/(1-\sin\theta))$ over the unit circle.  Both objects arise
from the same kernel $-\log(1-z)$ evaluated at $z=1$ (the
singularity) versus $z=e^{i\theta}\cdot 2/3$ (the circle average).
The residue identity~\eqref{eq:Phi-pole} would make this
geometric picture precise.
\end{remark}

\section{Further Directions}

We collect here several additional research directions of varying
depth, which we state without pursuing.

\subsection{Irrationality measure and convergence rates}

The rate at which $S_N\to S_{\mathrm{BBG}}$ is governed by
the irrationality of~$\pi$.  More precisely, the Erd\H{o}s--Tur\'an
bound in the proof of Theorem~\ref{thm:weyl} gives
$|S_{\mathrm{BBG}}-S_N|=O((2/3)^N \mu(\pi)\log N/N)$, where
$\mu(\pi)\le 7.103$~\cite{ZeilbergerZudilin2020} is the best known
irrationality measure.  Whether the precise value of $\mu(\pi)$
appears in the asymptotic expansion of $S_N$ remains open.

\subsection{Generalisation to irrational rotations}

For $\alpha/(2\pi)\notin\mathbb{Q}$, consider
\[
  S_{\mathrm{BBG}}(\alpha)
  \;:=\;
  \sum_{n=1}^{\infty}\frac{1}{n}
  \!\left(\frac{2+\sin(n\alpha)}{3}\right)^{\!n}.
\]
The proof of Theorem~\ref{thm:M} gives $M(\alpha)=\log 6$ for every
such~$\alpha$ (by the same Fubini argument, independent of the
arithmetic of $\alpha$).  It is natural to ask whether
$S_{\mathrm{BBG}}(\alpha)=\Ei(\log 3)$ for all irrational $\alpha/2\pi$,
or whether the value depends on Diophantine properties of $\alpha$.

\subsection{Ergodic and spectral interpretation}

The series $S_{\mathrm{BBG}}$ can be viewed as a weighted Birkhoff sum
\[
  S_{\mathrm{BBG}} \;=\;
  \sum_{n=1}^{\infty} F_n\!\left(T^n x_0\right),
  \quad T\colon x\mapsto x+\tfrac{1}{2\pi}\!\!\pmod{1},
  \quad x_0=0,
\]
where $F_n(x)=(1/n)(2+\sin(2\pi x))^n/3^n$.  Unlike classical
Birkhoff sums, the observable $F_n$ depends on~$n$, placing
$S_{\mathrm{BBG}}$ outside the scope of the Wiener--Wintner theorem
in its standard form.  Extending that theorem to non-stationary
observables of the form $F_n(x)=h(n,x)$ with $h(n,\cdot)$
converging in a suitable sense as $n\to\infty$ would be of
independent interest.

\subsection{Modular and motivic connections}

The function $J_n=(2/3)^{-n}I_n$ satisfies the asymptotic
$J_n\sim(3/2)^n/\sqrt{\pi n}$, and the generating function
$\sum_n J_n z^n$ is related to complete elliptic integrals.
Whether the special value $\Ei(\log 3)$ can be interpreted
in terms of periods of a motive over~$\mathbb{Q}$, in the
spirit of mixed Tate motives and multiple zeta values, is
an intriguing question that we leave for future investigation.


\end{document}